\documentclass{amsart}

\usepackage{amsmath}
\usepackage{amsthm}
\usepackage{amsfonts}
\usepackage{amssymb}
\usepackage{mathrsfs}
\usepackage[all]{xy}
\usepackage{latexsym}
\usepackage{enumerate}
\usepackage{tikz}
\usetikzlibrary{arrows}
\usetikzlibrary{decorations.markings}

\theoremstyle{plain}
\newtheorem{lemma}{Lemma}[section]
\newtheorem{prop}[lemma]{Proposition}
\newtheorem{theo}[lemma]{Theorem}
\theoremstyle{remark}
\newtheorem{rem}[lemma]{Remark}
\theoremstyle{definition}

\newcommand{\Su}{\mathbb{S}}
\newcommand{\Rel}{\mathcal{R}}
\newcommand{\mxl}{\textnormal{mxl}}
\newcommand{\mnl}{\textnormal{mnl}}
\newcommand{\op}{\textnormal{op}}
\newcommand{\B}{\mathcal{B}_X}
\newcommand{\Ss}{\mathcal{S}}
\newcommand{\X}{\mathcal{X}}
\newcommand{\K}{\mathcal{K}}
\newcommand{\N}{\mathbb{N}}
\newcommand{\Z}{\mathbb{Z}}
\newcommand{\tq}{\;/\;}

\title{Smallest homotopically trivial non-contractible spaces}

\author{Nicol\'as Cianci}
\address{Facultad de Ciencias Exactas y Naturales, Universidad Nacional de Cuyo and CONICET, Mendoza, Argentina.}
\email{nicocian@gmail.com}

\author{Miguel Ottina}
\address{Facultad de Ciencias Exactas y Naturales, Universidad Nacional de Cuyo, Mendoza, Argentina.}
\email{mottina@fcen.uncu.edu.ar}

\thanks{This research was partially supported by grant M015 of SeCTyP, UNCuyo. The first author was also partially supported by a CONICET doctoral fellowship.}

\begin{document}

\begin{abstract}
We characterize the homotopically trivial non-contractible topological spaces with the minimum number of points.
\end{abstract}

\maketitle

\section{Introduction}

Finite topological spaces have attracted increasing attention in the last years, principally from works by J. Barmak and G. Minian \cite{barmak2007minimal,barmak2008one,barmak2008simple,barmak2012strong}. One of the main interests of the theory of finite spaces is that they serve as models for weak homotopy types of compact polyhedra. More precisely, for every compact polyhedron $K$ there exists a finite T$_0$--space $\X(K)$ together with a weak homotopy equivalence $K\to\X(K)$ \cite{mccord1966singular}. Moreover, there is a functorial correspondence between finite T$_0$--spaces and finite posets \cite{alexandroff1937diskrete} which endows the theory of finite topological spaces with a natural combinatorial flavour. This allows the study of compact polyhedra (and often of general topological spaces) by means of combinatorial tools and gives a new insight into relevant topological questions \cite{barmak2011algebraic}. 

A natural problem of this theory is to find finite topological spaces with the minimum number of points which have certain weak homotopy type. One of the first questions of this type was asked by J.P. May in \cite{may2003finite}, where he conjectures that, for all $n\in\N$, the $n$--fold non-Hausdorff suspension of the $0$--sphere, denoted $\Su^n S^0$, is a \emph{minimal finite model} of the $n$--sphere, that is, a finite topological space which is weak homotopy equivalent to the $n$--sphere with the minimum possible cardinality. This question was answered positively by Barmak and Minian in \cite{barmak2007minimal}. Moreover, they also prove that $\Su^n S^0$ is the only minimal finite model of the $n$--sphere. In the same article they give a characterization of the minimal finite models of the finite graphs.

Further minimality questions were formulated in \cite{hardie2002nontrivial} and in \cite{barmak2011algebraic} regarding finite models of the real projective plane and the torus. These problems were solved in \cite{cianci2015splitting}, where a characterization of all the minimal finite models of these spaces is given.

Also, in \cite{rival1976fixed} and in \cite{barmak2011algebraic}, a homotopically trivial non-contractible space of nine points is given. Thus, another natural question to pose in this theory is whether this is the minimum number of points that a homotopically trivial non-contractible space can have \cite[problem 3.5.4]{may2016finite}. In this article we give an affirmative answer to this question showing that a homotopically trivial non-contractible space must have at least nine points. Moreover, we find all the homotopically trivial non-contractible spaces of nine points.

\section{Preliminaries}

In this section we will recall the basic notions of the theory of finite topological spaces and fix notation. For a comprehensive exposition on finite spaces the reader may consult \cite{barmak2011algebraic}. 

If $X$ is a finite topological space and $x\in X$, the intersection of all the open subsets of $X$ which contain $x$ is clearly an open subset and is denoted by $U_x$. Any finite T$_0$--space $X$ can be endowed with a partial order which is defined as follows: $x_1\leq x_2$ if and only if $U_{x_1}\subseteq U_{x_2}$. This defines a correspondence between finite T$_0$--spaces and finite posets which was first observed by Alexandroff \cite{alexandroff1937diskrete}. Moreover, under this correspondence continuous maps between finite T$_0$--spaces correspond to order-preserving morphisms between the respective posets. Hereafter, any finite T$_0$--space will be regarded also as a poset without further notice.

Let $X$ be a finite T$_0$--space and let $x\in X$. From the definition of the associated partial order it follows that $U_x=\{a\in X \tq a\leq x\}$. In a similar way, the smallest closed set which contains $x$ is $\overline{\{x\}}=\{a\in X \tq a\geq x\}$ and is denoted by $F_x$. It is also standard to define $\widehat{U}_x=\{a\in X \tq a< x\}$, $\widehat{F}_x=\{a\in X \tq a> x\}$, $C_x=U_x\cup F_x$ and $\widehat{C}_x=C_x-\{x\}$. We say that the point $x\in X$ is an \emph{up beat point} (resp. \emph{down beat point}) of $X$ if the subposet $\widehat{F}_x$ has a minimum (resp. if the subposet $\widehat{U}_x$ has a maximum) \cite{stong1966finite,may2003finite,barmak2011algebraic}. 

Stong proves in \cite{stong1966finite} that if $x$ is a beat point of $X$ then $X-\{x\}$ is a strong deformation retract of $X$ and that two finite T$_0$--spaces are homotopy equivalent if and only if one obtains homeomorphic spaces after successively removing their beat points. Using the results of Stong it is easy to prove that a finite T$_0$--space which has a maximum or a minimum is contractible.

If $X$ is a finite T$_0$--space, $\K(X)$ will denote the \emph{order complex} of $X$, that is, the simplicial complex of the non-empty chains of $X$. Also, $X^\op$ will denote the poset $X$ with the inverse order and will be called the \emph{opposite} space of $X$.

McCord proves in \cite{mccord1966singular} that if $X$ is a finite T$_0$--space then there exists a weak homotopy equivalence from the geometric realization of $\K(X)$ to $X$. In particular, any finite T$_0$--space is weak homotopy equivalent to its opposite space since their order complexes coincide. Note also that the aforementioned result of McCord implies that the singular homology groups of a finite T$_0$--space $X$ are isomorphic to the simplicial homology groups of $\K(X)$.

The \emph{non-Hausdorff suspension} of a topological space $X$ is defined as the space $\Su X$ whose underlying set is $X\amalg \{+,-\}$ and whose open sets are those of $X$ together with $X\cup\{+\}$, $X\cup\{-\}$ and $X\cup\{+,-\}$ \cite{mccord1966singular}. Note that, if $X$ is a finite T$_0$--space then the partial order in $\Su X$ is induced by the partial order of $X$ together with the relations $x\leq +$ and $x\leq -$ for all $x\in X$. McCord proves that for every topological space $X$ there exists a weak homotopy equivalence between the suspension of $X$ and $\Su X$ \cite{mccord1966singular}. As an example, he shows that, for all $n\in\N$, the $n$--sphere $S^n$ is weak homotopy equivalent to the $n$--fold non-Hausdorff suspension of the $0$--sphere $S^0$. Observe that $\Su^n S^0$ is a finite T$_0$--space of $2n+2$ points.

May asked in \cite{may2003finite} if $\Su^n S^0$ was the smallest space which is weak homotopy equivalent to the $n$--sphere. This question was answered by Barmak and Minian in \cite{barmak2007minimal}. More precisely, they proved the following theorem from which the affirmative answer to May's question follows.

\begin{theo} \label{theo_height_and_cardinality}
Let $X \neq\ast$ be a finite topological space without beat points. Then $X$ has at least $2h(X)+2$ points.
Moreover, if $X$ has exactly $2h(X)+2$ points, then it is homeomorphic to $\Su^{h(X)} S^0$.
\end{theo}

In \cite{cianci2014homology} we studied the homology groups of finite T$_0$--spaces obtaining several results and applications. Among them we mention the following proposition, which will be needed later. In what follows, homology will always mean homology with integer coefficients. Thus, the group of coefficients will be omitted from the notation.

\begin{prop}\label{prop_relative_homology}
Let $X$ be a finite $T_0$--space and let $D$ be an antichain in $X$. Then $H_n(X,X-D)\cong\bigoplus\limits_{x\in D}\tilde{H}_{n-1}(\hat{C}_x)$ for every $n\in \Z$.
\end{prop}

If $X$ is a finite T$_0$--space, $\mxl(X)$ and $\mnl(X)$ will denote the subsets of maximal and minimal points of $X$ respectively. The following remark states some simple facts concerning the maximal and minimal points of a finite T$_0$--space. The first two items already appeared in \cite{cianci2015splitting}.

\begin{rem} \label{rem_mxl_mnl}
\ 

\begin{enumerate}
\item Let $X$ be a connected and finite T$_0$--space with more than one point. Then $\mxl(X)\cap\mnl(X)=\varnothing$.
\item Let $X$ be a finite T$_0$--space without beat points. If $a\in X-\mxl(X)$ then $\#(\widehat F_a \cap \mxl(X) )\geq 2$. Similarly, if $b\in X-\mnl(X)$ then $\#(\widehat U_b \cap \mnl(X) )\geq 2$.
\item Let $X$ be a finite T$_0$--space without beat points. If $\#\mxl(X)=2$ then $X$ is homeomorphic to $\Su(X-\mxl(X))$.
\end{enumerate}
\end{rem}

Finally, if $X$ is a finite T$_0$--space, we define the \emph{height} of $X$ as 
\begin{displaymath}
h(X)=\max\{\#c-1\;/\; \textnormal{$c$ is a chain of $X$}\}. 
\end{displaymath}
Note that $h(X)=\dim \K(X)$.

\section{Results}

In this section we will prove that a homotopically trivial non-contractible space must have at least nine points. In addition, we will find all the homotopically trivial non-contractible spaces of nine points. 

In figure \ref{fig_ht_non-contractible} we exhibit a homotopically trivial non-contractible T$_0$--space of nine points, which was also considered in \cite[Figure 2]{rival1976fixed} and in \cite[Example 4.3.3]{barmak2011algebraic}. Observe that this space is not contractible since it does not have beat points and is homotopically trivial since the geometric realization of its order complex is contractible. 

\begin{figure}[h!]
\begin{tikzpicture}[x=3cm,y=3cm]
	\tikzstyle{every node}=[font=\footnotesize]
	
	\foreach \x in {1,...,3} \draw (0.5*\x,1) node(a\x){$\bullet$} node[above=1]{$a_{\x}$};
	\foreach \x in {1,...,3} \draw (0.5*\x,0.5) node(b\x){$\bullet$} node[right=1]{$b_{\x}$};
	\foreach \x in {1,...,3} \draw (0.5*\x,0) node(c\x){$\bullet$} node[below=1]{$c_{\x}$};
	
	\foreach \x in {1,2} \draw (a1)--(b\x);
	\foreach \x in {1,3} \draw (a2)--(b\x);
	\foreach \x in {2,3} \draw (a3)--(b\x);
		
	\foreach \x in {1,2} \draw (c1)--(b\x);
	\foreach \x in {1,2,3} \draw (c2)--(b\x);
	\foreach \x in {2,3} \draw (c3)--(b\x);
\end{tikzpicture}
\caption{Homotopically trivial non-contractible space of 9 points.}
\label{fig_ht_non-contractible}
\end{figure}
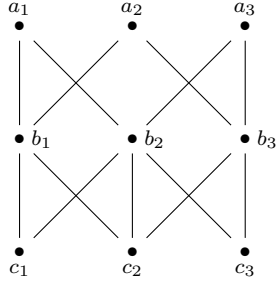

We will give now several lemmas which will be useful to prove the main results of this article.

\begin{lemma} \label{lemma_card_Ua_Ub}
Let $X$ be a finite T$_0$--space without beat points. Let $a,b\in X$ with $a>b$. Then $\#\widehat{U}_a\geq \#\widehat{U}_b+2$ and $\#\widehat{F}_b\geq \#\widehat{F}_a+2$.
\end{lemma}

\begin{proof}
Note that $\widehat U_a \supseteq U_b$ since $b<a$. And since $a$ is not a beat point of $X$ we obtain that $\widehat{U}_a \supsetneq U_b$. Therefore $\# \widehat U_a\geq \# U_b + 1=\#\widehat{U}_b+2$.

Applying this result to $X^\op$ we obtain that $\#\widehat{F}_b\geq \#\widehat{F}_a+2$.
\end{proof}

\begin{lemma} \label{lemma_subspace_with_non_trivial_H2}
Let $X$ be a finite T$_0$--space such that $h(X)=2$. If $X$ is homotopically trivial then $H_2(A)=0$ for all subspaces $A\subseteq X$.
\end{lemma}

\begin{proof}
Let $A$ be a subspace of $X$. Note that $H_3(X,A)=0$ since $h(X)=2$. The result then follows from the exact sequence $H_3(X,A) \longrightarrow H_2(A) \longrightarrow H_2(X)$.
\end{proof}

\begin{lemma} \label{lemma_no_subspace_finite_S2}
Let $X$ be a homotopically trivial finite T$_0$--space such that $h(X)=2$. Let $b,b'\in X-(\mxl(X)\cup\mnl(X))$. If $\#(\widehat{F}_{b}\cap \widehat{F}_{b'})\geq 2$ then $\#(\widehat{U}_{b}\cap \widehat{U}_{b'})\leq 1$.
\end{lemma}

\begin{proof}
Note that $\{b,b'\}$ must be an antichain. Suppose that $\#(\widehat{U}_{b}\cap \widehat{U}_{b'})\geq 2$. Then there exist distinct elements $a,a'\in \widehat{F}_{b}\cap \widehat{F}_{b'}$ and distinct elements $c,c'\in \widehat{U}_{b}\cap \widehat{U}_{b'}$. Note that $\{a,a'\}\subseteq \mxl(X)$ and that $\{c,c'\}\subseteq \mnl(X)$. Hence the subspace $A=\{a,a',b,b',c,c'\}$ is homeomorphic to $\Su^2 S^0$ and then $H_2(A)\neq 0$, contradicting \ref{lemma_subspace_with_non_trivial_H2}. Thus $\#(\widehat{U}_{b}\cap \widehat{U}_{b'})\leq 1$.
\end{proof}

\begin{lemma} \label{lemma_homotopically_trivial_suspension}
Let $X$ be a topological space such that, for some $x_0\in X$, $\pi_1(X,x_0)$ is not a non-trivial perfect group. If $\Su X$ is homotopically trivial then $X$ is homotopically trivial.
\end{lemma}

\begin{proof}
Let $\Sigma X$ be the suspension of $X$. Since $\Sigma X$ is weak homotopy equivalent to $\Su X$ we obtain that $\Sigma X$ is homotopically trivial. Thus $H_n(\Sigma X)=0$ for all $n\in\N$. Hence $X$ is path-connected and $H_n(X)=0$ for all $n\in\N$. Thus $\pi_1(X,x_0)$ is a perfect group. Therefore $\pi_1(X,x_0)$ must be the trivial group. The result then follows from Hurewicz's theorem.
\end{proof}

The following proposition shows that the height of a homotopically trivial non-contractible finite T$_0$--space must be greater than 1.

\begin{prop} \label{prop_ht_of_height_1}
Let $X$ be a homotopically trivial finite T$_0$--space with $h(X)\leq 1$. Then $X$ is contractible.
\end{prop}

\begin{proof}
Suppose that $X$ is not contractible. Without loss of generality, we may assume that $X$ does not have beat points. Let $E$ denote the set of edges of the Hasse diagram of $X$. Let $\Rel\subseteq X\times E$ be the relation defined by $x\Rel a$ if and only if the point $x$ belongs to the edge $a$.
Note that $\#\Rel=2\#E$. And since $X$ is a path-connected space, $\#X\geq 2$ and $X$ does not have beat points, from \ref{rem_mxl_mnl} we obtain that $\#\Rel\geq 2\#X$. Thus $\#E\geq \#X$. But $1=\chi(X)=\# X - \#E \leq 0$ which entails a contradiction.
\end{proof}

We will prove now one of the main results of this article.

\begin{theo} \label{theo_ht_non_contractible_1}
Let $X$ be a homotopically trivial non-contractible topological space. Then $\#X \geq 9$.
\end{theo}

\begin{proof}
By the works of McCord \cite{mccord1966singular} and Stong \cite{stong1966finite} we may assume that $X$ is a finite T$_0$--space without beat points.

By \ref{theo_height_and_cardinality}, if $h(X)\geq 4$ then $\#X\geq 10$. And by \ref{prop_ht_of_height_1}, $h(X)\geq 2$. Thus $h(X)=2$ or $h(X)=3$.

\underline{Case 1}: $h(X)=3$. By \ref{theo_height_and_cardinality}, $\#X\geq 8$. If $\#X=8$ then, again by \ref{theo_height_and_cardinality}, $X$ is homeomorphic to $\Su^{3} S^0$ which is weak equivalent to $S^3$. But this is a contradiction since $X$ is homotopically trivial. Thus $\#X\geq 9$.

\underline{Case 2}: $h(X)=2$. If $\#\mxl(X)=2$ then, from \ref{rem_mxl_mnl}, we obtain that $X=\Su Y$ with $Y=X-\mxl(X)$ and $h(Y)=1$. Thus $\pi_1(Y,y_0)$ is a free group for all $y_0\in Y$. By \ref{lemma_homotopically_trivial_suspension}, $Y$ is homotopically trivial. Since $X$ does not have beat points, the same holds for $Y$. But this is not possible by \ref{prop_ht_of_height_1}. Thus $\#\mxl(X)\geq 3$.

Working with $X^\op$ in a similar way, we obtain that $\#\mnl(X)\geq 3$. Let $\B=X-(\mxl(X)\cup\mnl(X))$. Since $h(X)=2$ we obtain that $\B$ is a non-empty antichain of $X$. If $\#\B\geq 3$ then $\# X \geq 9$. Thus we may assume that $\#\B\leq 2$.

\underline{Case 2.1}: $\#\B = 1$. Suppose that $\B=\{b\}$. Note that $\widehat{F}_b\subseteq \mxl(X)$ and $\widehat{U}_b\subseteq \mnl(X)$. Let $\alpha_b=\# \widehat{F}_b$ and $\beta_b=\#\widehat{U}_b$. By \ref{rem_mxl_mnl}, $\alpha_b \geq 2$ and $\beta_b\geq 2$.

Since $X$ is path-connected, from the first item of \ref{rem_mxl_mnl} we obtain that $\mxl(X)\cap\mnl(X)=\varnothing$ and thus $\# \widehat{U}_a\geq 2$ for all $a\in\mxl(X)$ by the second item of \ref{rem_mxl_mnl}. From \ref{lemma_card_Ua_Ub} we obtain that $\# \widehat{U}_a\geq \beta_b + 2$ for all $a\in\widehat{F}_b$.

Let $l$ denote the number of $1$--chains of $X$ and let $m=\#\mxl(X)$. We have that
\begin{displaymath}
l=\# \widehat{U}_b + \!\! \sum_{a\in \mxl(X)}\!\!\!\!\!\# \widehat{U}_a \geq \beta_b + \alpha_b(\beta_b+2)+ 2(m-\alpha_b) = \beta_b + \alpha_b\beta_b + 2m\geq \alpha_b\beta_b + 8.
\end{displaymath}
On the other hand, note that the number of $2$--chains of $X$ is $\alpha_b\beta_b$.

Hence
\begin{displaymath}
1=\chi(X)=\#X-l+\alpha_b\beta_b \leq \#X - (\alpha_b\beta_b + 8) + \alpha_b\beta_b = \#X -8 \ .
\end{displaymath}
Therefore $\#X\geq 9$.

\underline{Case 2.2}: $\#\B = 2$. If $\#\mxl(X)\geq 4$ or $\#\mnl(X)\geq 4$ then $\#X\geq 9$. Thus we may assume that $\#\mxl(X)=\#\mnl(X)=3$. As above, note that $\widehat{F}_b\subseteq \mxl(X)$ and $\widehat{U}_b\subseteq \mnl(X)$ for all $b\in \B$ since $\B$ is an antichain.

First, we will prove that $\# \widehat{U}_b \neq 3$ for all $b\in \B$. Let $b\in \B$ and suppose that $\# \widehat{U}_b = 3$. Then $\widehat{U}_b=\mnl(X)$. Let $b'$ be the element of $\B-\{b\}$. We claim that $\widehat{F}_b\subseteq \widehat{F}_{b'}$. Indeed, let $a\in \widehat{F}_b$. Then $\widehat{U}_a \supseteq U_b$ and since $a$ is not a beat point of $X$ we obtain that $\widehat{U}_a \supsetneq U_b$. But since $U_b \supseteq \mnl(X)$ and $a\in\mxl(X)$ it follows that $b'\in \widehat{U}_a$ and hence $a\in \widehat{F}_{b'}$.

By \ref{rem_mxl_mnl}, $\# \widehat{F}_b \geq 2$ and $\# \widehat{U}_{b'} \geq 2$. Thus $\#(\widehat{F}_b\cap \widehat{F}_{b'})\geq 2$ and $\#(\widehat{U}_b\cap \widehat{U}_{b'})\geq 2$ contradicting \ref{lemma_no_subspace_finite_S2}. Therefore $\# \widehat{U}_b \neq 3$ for all $b\in \B$.

Hence $\# \widehat{U}_b = 2$ for all $b\in \B$ by \ref{rem_mxl_mnl}. In a similar way, we obtain that $\# \widehat{F}_b = 2$ for all $b\in \B$.

Let $\mxl(X)=\{a_1,a_2,a_3\}$, $\B=\{b_1,b_2\}$ and $\mnl(X)=\{c_1,c_2,c_3\}$. Without loss of generality, we may assume that $\widehat{F}_{b_1}=\{a_1,a_2\}$ and $\widehat{U}_{b_1}=\{c_1,c_2\}$. We will prove that $\#\widehat{U}_{a_1}+\#\widehat{U}_{a_2}+\#\widehat{U}_{a_3}\geq 12$ analizing two cases: $b_2<a_3$ and $b_2\nless a_3$.

If $b_2<a_3$ then $\#\widehat{U}_{a_3}\geq\#\widehat{U}_{b_2}+2= 4$ by \ref{lemma_card_Ua_Ub}. And since $b_1<a_1$ and $b_1<a_2$, we also obtain that $\#\widehat{U}_{a_1}\geq 4$ and $\#\widehat{U}_{a_2}\geq 4$ by \ref{lemma_card_Ua_Ub}. Thus $\#\widehat{U}_{a_1}+\#\widehat{U}_{a_2}+\#\widehat{U}_{a_3}\geq 12$.

If $b_2\nless a_3$ then $\widehat{F}_{b_2}=\{a_1,a_2\}=\widehat{F}_{b_1}$. By \ref{lemma_no_subspace_finite_S2}, $\widehat{U}_{b_2}\neq \widehat{U}_{b_1}=\{c_1,c_2\}$. Hence $c_3<b_2$. Thus $\widehat{U}_{a_1}=\widehat{U}_{a_2}=\{b_1,b_2,c_1,c_2,c_3\}$ and since $\#\widehat{U}_{a_3}\geq 2$ by \ref{rem_mxl_mnl} we obtain that $\#\widehat{U}_{a_1}+\#\widehat{U}_{a_2}+\#\widehat{U}_{a_3}\geq 12$.

Therefore $\#\widehat{U}_{a_1}+\#\widehat{U}_{a_2}+\#\widehat{U}_{a_3}\geq 12$ in any case. As above, let $l$ denote the number of $1$--chains of $X$. Then, applying \ref{rem_mxl_mnl}, we get
\begin{displaymath}
l=\#\widehat{U}_{a_1}+\#\widehat{U}_{a_2}+\#\widehat{U}_{a_3}+\#\widehat{U}_{b_1}+\#\widehat{U}_{b_2}\geq 12 + 2 + 2 = 16.
\end{displaymath}
Now, note that the number of $2$--chains of $X$ is $\#\widehat{F}_{b_1}\#\widehat{U}_{b_1}+\#\widehat{F}_{b_2}\#\widehat{U}_{b_2}=8$. Thus
\begin{displaymath}
\chi(X)= \#X-l+8 \leq 8 - 16 + 8 = 0
\end{displaymath}
and hence the space $X$ is not homotopically trivial.
\end{proof}

As a corollary of the previous theorem we obtain that the space of figure \ref{fig_ht_non-contractible} is a homotopically trivial non-contractible space with the minimum possible number of points. In the following theorem we find all the homotopically trivial non-contractible spaces of this minimum number of points.

\begin{theo}
Let $X$ be a homotopically trivial non-contractible topological space such that $\#X = 9$. Then $X$ is homeomorphic to either the space of figure \ref{fig_ht_non-contractible} or its opposite.
\end{theo}

\begin{proof}
By \ref{theo_ht_non_contractible_1} we may assume that $X$ is a finite T$_0$--space without beat points. By \ref{rem_mxl_mnl}, $\#\mxl(X)\geq 2$. If $\#\mxl(X)= 2$, $X$ is homeomorphic to $\Su(X-\mxl(X))$ by \ref{rem_mxl_mnl}. Note that $X-\mxl(X)$ does not have beat points. Since $\#(X-\mxl(X))=7$, from \cite[theorem 5.7]{cianci2015splitting} we obtain that $\pi_1(X-\mxl(X),x_0)$ is a free group for all $x_0\in X-\mxl(X)$. Thus $X-\mxl(X)$ is homotopically trivial by \ref{lemma_homotopically_trivial_suspension}, which contradicts \ref{theo_ht_non_contractible_1}. Therefore $\#\mxl(X)\geq 3$. And applying this argument to $X^\op$ we obtain that $\#\mnl(X)\geq 3$.

Let $\B=X-(\mxl(X)\cup\mnl(X))$. From the first item of \ref{rem_mxl_mnl} it follows that $\#\B\leq 3$. By \ref{prop_ht_of_height_1}, $h(X)\geq 2$. Thus $\B\neq\varnothing$.

We will analize three cases which correspond to the possible cardinalities of the subset $\B$.

\underline{Case 1}: $\# \B=1$. Let $b$ be the only element of $B$. Let $n=\#\mnl(X)$, $\alpha=\#\widehat{F}_b$ and $\beta=\#\widehat{U}_b$. By \ref{rem_mxl_mnl}, $\alpha\geq 2$ and $\beta\geq 2$. Let $\Rel\subseteq X\times X$ be the order relation of $X$ and let $\Ss=\Rel\cap(\mnl(X)\times\mxl(X))$.

Let $\Ss_1=\Rel\cap(\widehat{U}_b\times\widehat{F}_b)$, $\Ss_2=\Rel\cap(\widehat{U}_b\times(\mxl(X)-\widehat{F}_b))$ and $\Ss_3=\Rel\cap((\mnl(X)-\widehat{U}_b)\times\mxl(X))$. Clearly, $\Ss_1$, $\Ss_2$ and $\Ss_3$ are pairwise disjoint and $\Ss=\Ss_1\cup \Ss_2\cup \Ss_3$.

Note that $\#\Ss_1=\alpha\beta$. Also, if $z\in \widehat{U}_b$ then $\widehat{F}_z\supsetneq F_b$ since $z$ is not a beat point of $X$. Hence $\widehat{F}_z\cap (\mxl(X)-\widehat{F}_b)\neq\varnothing$ for all $z\in \widehat{U}_b$. Thus $\#\Ss_2\geq \beta$. On the other hand, from \ref{rem_mxl_mnl} we obtain that $\#\Ss_3\geq 2\#(\mnl(X)-\widehat{U}_b)=2(n-\beta)$.

Thus 
\begin{displaymath}
\#\Ss\geq \alpha\beta + \beta + 2(n-\beta)\geq \alpha\beta + 2 + 2(n-\beta) \ .
\end{displaymath}

Proceeding in a similar way we also obtain that
\begin{displaymath}
\#\Ss \geq \alpha\beta + 2 + 2 \# (\mxl(X)-\widehat{F}_b)= \alpha\beta + 2 + 2 (8-n-\alpha) \ .
\end{displaymath}
Hence
\begin{displaymath}
\#\Ss \geq \alpha\beta + 2 + 2\max\{n-\beta,8-n-\alpha\}\geq \alpha\beta + 2 + (n-\beta) + (8-n-\alpha) = \alpha\beta+10-\alpha-\beta \ .
\end{displaymath}
Thus we obtain that
\begin{displaymath}
\chi(X)=9-(\#\Ss+\alpha+\beta)+\alpha\beta \leq 9 - \alpha\beta - 10 + \alpha\beta = -1 \ .
\end{displaymath}
Hence $X$ is not homotopically trivial. Therefore this case is not possible.

\underline{Case 2}: $\# \B=2$. Without loss of generality we may assume that $\#\mxl(X)=3$ and $\#\mnl(X)=4$. Let $b_1$ and $b_2$ be the elements of $\B$. 

We will prove that $\B$ is an antichain. Indeed, if $b_1<b_2$ then $\# \widehat{F}_{b_1} \geq \#\widehat{F}_{b_2}+2 \geq 4$ by \ref{lemma_card_Ua_Ub} and \ref{rem_mxl_mnl}. Thus $\widehat{F}_{b_1}=\{b_2\}\cup\mxl(X)$.
Let $c\in\mnl(X)\cap U_{b_1}$. Then $F_{b_1}\subseteq \widehat{F}_c \subseteq X-\mnl(X) = \B\cup\mxl(X) = F_{b_1}$. Hence $\widehat{F}_c=F_{b_1}$ and then $c$ is a beat point of $X$, which contradicts our assumptions. Therefore $\B$ must be an antichain.

Hence $h(X)=2$. For $j\in\{1,2\}$, let $\alpha_j=\#\widehat{F}_{b_j}$ and $\beta_j=\#\widehat{U}_{b_j}$. Note that $\alpha_j\geq 2$ and $\beta_j\geq 2$ for all $j\in\{1,2\}$ by \ref{rem_mxl_mnl}. As in the previous case, let $\Rel\subseteq X\times X$ be the order relation of $X$ and let $\Ss=\Rel\cap(\mnl(X)\times\mxl(X))$. Thus
\begin{displaymath}
\chi(X)=9-(\alpha_1+\beta_1+\alpha_2+\beta_2+\#\Ss)+\alpha_1\beta_1+\alpha_2\beta_2=7-\#\Ss+\sum_{j=1}^2 (\alpha_j-1)(\beta_j-1)\ .
\end{displaymath}
We will analyze two subcases.

\underline{Case 2.1}: $\#(\widehat{F}_{b_1}\cap \widehat{F}_{b_2})=1$. In this case, $\alpha_1=\alpha_2=2$ and $\widehat{F}_{b_1}\cup \widehat{F}_{b_2}=\mxl(X)$. Let $a$ be the only point of $\widehat{F}_{b_2}-\widehat{F}_{b_1}$.

We will prove now that $\beta_2\leq 3$. Suppose that $\beta_2>3$. Then $\beta_2=4$ and $\widehat{U}_{b_2}=\mnl(X)$. Thus $\widehat{U}_{a}=U_{b_2}$ and hence $a$ is a beat point of $X$, which entails a contradiction. Therefore $\beta_2\leq 3$.

We will prove now that $\mxl(X)\subseteq F_c$ for all $c\in \widehat{U}_{b_1}$. Let $c\in \widehat{U}_{b_1}$. If $b_2>c$ then $F_c\supseteq F_{b_1}\cup F_{b_2}\supseteq \mxl(X)$. If $b_2 \ngtr c$ then $F_{b_1}\subseteq \widehat{F}_c \subseteq \{b_1\}\cup\mxl(X)= F_{b_1}\cup \{a\}$. And since $c$ is not a beat point of $X$, we obtain that $\widehat{F}_c = F_{b_1}\cup \{a\}$ and hence $\mxl(X)\subseteq F_c$.

Since $\mxl(X)\subseteq F_c$ for all $c\in \widehat{U}_{b_1}$ and applying \ref{rem_mxl_mnl}, we obtain that 
\begin{displaymath}
\#\Ss\geq 3\#\widehat{U}_{b_1} + 2 \#(\mnl(X)-\widehat{U}_{b_1})= 3 \beta_1 + 2 (4-\beta_1)=\beta_1+8 \ .
\end{displaymath}
Thus
\begin{displaymath}
\chi(X)=7-\#\Ss+\sum_{j=1}^2 (\alpha_j-1)(\beta_j-1)\leq 7 - \beta_1 - 8 + \beta_1-1+\beta_2-1=\beta_2-3 \leq 0 
\end{displaymath}
and hence $X$ is not homotopically trivial.

\underline{Case 2.2}: $\#(\widehat{F}_{b_1}\cap \widehat{F}_{b_2})\geq 2$. By \ref{lemma_no_subspace_finite_S2}, $\#(\widehat{U}_{b_1}\cap \widehat{U}_{b_2})\leq 1$.

Since $\beta_1\geq 2$ and $\beta_2\geq 2$ we obtain that $\widehat{U}_{b_1}-\widehat{U}_{b_2}\neq\varnothing$ and $\widehat{U}_{b_2}-\widehat{U}_{b_1}\neq\varnothing$. Let $c \in \widehat{U}_{b_1}-\widehat{U}_{b_2}$. We have that $F_{b_1}\subseteq\widehat{F}_{c}\subseteq \{b_1\}\cup\mxl(X)$. And since $c$ is not a beat point of $X$ we obtain that $\widehat{F}_{c}\neq F_{b_1}$. Therefore $\mxl(X)\nsubseteq F_{b_1}$ which implies that $\alpha_1 = 2$. Thus $\#(\mxl(X)-F_{b_1})= 1$ and since $\widehat{F}_{c}\neq F_{b_1}$ we obtain that $\mxl(X)\subseteq F_c$.

In a similar way we obtain that $\alpha_2 = 2$ and that $\mxl(X)\subseteq F_c$ for all $c\in \widehat{U}_{b_2}-\widehat{U}_{b_1}$. Thus, applying \ref{rem_mxl_mnl}, we obtain that $\#\Ss\geq 3+3+2+2 = 10$.

Hence
\begin{displaymath}
\chi(X)=7-\#\Ss+\sum_{j=1}^2 (\alpha_j-1)(\beta_j-1)\leq 7 -10 + \beta_1-1+\beta_2-1=\beta_1+\beta_2-5 \leq 0 
\end{displaymath}
where the last inequality follows from the fact that $\#(\widehat{U}_{b_1}\cap \widehat{U}_{b_2})\leq 1$. Therefore $X$ is not homotopically trivial.

\underline{Case 3}: $\# \B=3$. Note that $\#\mxl(X)=\#\mnl(X)=3$. Let $b_1$, $b_2$ and $b_3$ be the elements of $\B$.

Suppose that $\B$ is a chain. Without loss of generality we may assume that $b_1<b_2<b_3$. Applying \ref{lemma_card_Ua_Ub} and \ref{rem_mxl_mnl} we obtain that
\begin{displaymath}
\# \widehat{F}_{b_1}\geq \# \widehat{F}_{b_2} +2 \geq \# \widehat{F}_{b_3} + 4 \geq 6
\end{displaymath}
but this can not be possible since $\widehat{F}_{b_1}\subseteq \{b_2,b_3\}\cup \mxl(X)$. Thus $\B$ is not a chain.

Hence, without loss of generality, we may assume that $b_1$ and $b_3$ are incomparable.

Suppose now that $\B$ is not an antichain. Without loss of generality we may assume that $b_2$ and $b_1$ are comparable elements, and considering $X^\op$ if necessary, we may suppose that $b_2>b_1$. Note that $b_2\nless b_3$ since $\B$ is not a chain. Hence $\widehat{F}_{b_2}\subseteq \mxl(X)$. By \ref{rem_mxl_mnl}, $\#\widehat{F}_{b_2}\geq 2$. Let $a_1$ and $a_2$ be distinct elements of $\widehat{F}_{b_2}$ and let $a_3$ be the remaining maximal element of $X$. By \ref{lemma_card_Ua_Ub}, $\#\widehat{F}_{b_1}\geq \#\widehat{F}_{b_2}+2 \geq 4$ and since $b_1$ and $b_3$ are incomparable it follows that $\widehat{F}_{b_1}=\{b_2,a_1,a_2,a_3\}$ and $\widehat{F}_{b_2}=\{a_1,a_2\}$.

Now, observe that $\widehat{F}_{b_3}\subseteq \{b_2,a_1,a_2,a_3\}$ since $b_1$ and $b_3$ are incomparable. We claim that $\widehat{F}_{b_3}$ is not path-connected. Indeed, if $b_3<b_2$ then proceeding as in the previous paragraph we obtain that $\widehat{F}_{b_3}=\{b_2,a_1,a_2,a_3\}$ which is not path-connected. And if $b_3\nless b_2$ then $\widehat{F}_{b_3}\subseteq\mxl(X)$ and hence $\widehat{F}_{b_3}$ is a discrete subspace and $\#\widehat{F}_{b_3}\geq 2$ by \ref{rem_mxl_mnl}. Thus $\widehat{F}_{b_3}$ is not path-connected.

By \ref{rem_mxl_mnl}, $\# (\widehat{U}_{b_1}\cap\mnl(X))\geq 2$. Let $c_1$ and $c_2$ be distinct elements of  $\widehat{U}_{b_1}\cap\mnl(X)$ and let $c_3$ be the remaining minimal element of $X$. Since $c_1<b_1$ and $c_1$ is not a beat point of $X$ we obtain that $\{b_1,b_2,a_1,a_2,a_3\}=F_{b_1}\subsetneq\widehat{F}_{c_1}\subseteq X-\mnl(X)$. Hence $\widehat{F}_{c_1}=\{b_1,b_2,b_3,a_1,a_2,a_3\}$. In a similar way, $\widehat{F}_{c_2}=\{b_1,b_2,b_3,a_1,a_2,a_3\}$.

From \ref{prop_relative_homology} we obtain that 
$H_1(\widehat{F}_{c_1})\cong H_1(\widehat{F}_{c_1},F_{b_1}) \cong \widetilde{H}_0(\widehat{F}_{b_3})\neq 0$
since $\widehat{F}_{b_3}$ is not path-connected.

Applying \ref{prop_relative_homology} again we obtain that 
$H_2(X)\cong H_2(X,F_{c_2})\cong H_1(\widehat{F}_{c_1})\oplus H_1(\widehat{F}_{c_3}) \neq 0$
and hence $X$ is not homotopically trivial. Therefore $\B$ must be an antichain.

Hence $h(X)=2$. For $j\in\{1,2,3\}$, let $\alpha_j=\#\widehat{F}_{b_j}$ and $\beta_j=\#\widehat{U}_{b_j}$. Note that $\alpha_j\geq 2$ and $\beta_j\geq 2$ for all $j\in\{1,2,3\}$ by \ref{rem_mxl_mnl}. As in the previous cases, let $\Rel\subseteq X\times X$ be the order relation of $X$ and let $\Ss=\Rel\cap(\mnl(X)\times\mxl(X))$. Thus
\begin{displaymath}
1=\chi(X)=9-\left(\sum_{j=1}^3\alpha_j+\sum_{j=1}^3\beta_j +\#\Ss\right)+\sum_{j=1}^3\alpha_j\beta_j = 6-\#\Ss+\sum_{j=1}^3 (\alpha_j-1)(\beta_j-1)\ .
\end{displaymath}
Hence
\begin{displaymath}
\#\Ss= 5+\sum_{j=1}^3 (\alpha_j-1)(\beta_j-1)\geq 8\ .
\end{displaymath}
Thus $\#\Ss=8$ or $\#\Ss=9$ and hence $\sum\limits_{j=1}^3 (\alpha_j-1)(\beta_j-1)\in\{3,4\}$.

Therefore at least five of the numbers $\alpha_1$, $\alpha_2$, $\alpha_3$, $\beta_1$, $\beta_2$ and $\beta_3$ are equal to $2$ and the remaining one might be $2$ or $3$.

Without loss of generality and considering $X^\op$ if necessary, we may assume that $\alpha_1=\alpha_2=\alpha_3=2$.

\underline{Claim 1}: $\#(\widehat{F}_{b_k}\cap\widehat{F}_{b_l})=1$ for all $k,l\in\{1,2,3\}$ with $k\neq l$. 

Suppose that there exist $k,l\in\{1,2,3\}$ with $k\neq l$ such that $\#(\widehat{F}_{b_k}\cap\widehat{F}_{b_l})=2$. Without loss of generality we may assume that $\#(\widehat{F}_{b_1}\cap\widehat{F}_{b_2})=2$. 

Thus $\#(\widehat{U}_{b_1}\cap \widehat{U}_{b_2})\leq 1$ by \ref{lemma_no_subspace_finite_S2}. And since $\#\mnl(X)=3$ we obtain that $\#(\widehat{U}_{b_1}\cap \widehat{U}_{b_2})=1$. Hence $\beta_1=\beta_2=2$. Thus $\widehat{U}_{b_1}\cup \widehat{U}_{b_2}=\mnl(X)$.

Let $a_1$ and $a_2$ be the elements of $\widehat{F}_{b_1}\cap\widehat{F}_{b_2}$ and let $a_3$ be the remaining maximal element of $X$. Note that $U_{a_1}\supseteq \mnl(X)$ and $U_{a_2}\supseteq \mnl(X)$.

If $\{a_1,a_2\}\subseteq F_{b_3}$ then $\#(\widehat{U}_{b_1}\cap \widehat{U}_{b_3})\leq 1$ and $\#(\widehat{U}_{b_2}\cap \widehat{U}_{b_3})\leq 1$ by \ref{lemma_no_subspace_finite_S2}. Thus $\#(\widehat{U}_{b_1}\cap \widehat{U}_{b_3})=1$ and $\#(\widehat{U}_{b_2}\cap \widehat{U}_{b_3})=1$. Hence $|\K(X-\{a_3\})|$ is homeomorphic to $S^2$ contradicting \ref{lemma_subspace_with_non_trivial_H2}. Thus $\{a_1,a_2\}\nsubseteq F_{b_3}$.

Hence, $a_3\in \widehat{F}_{b_3}$. Since $\alpha_1=\alpha_2=2$, we obtain that $b_1\nless a_3$ and $b_2\nless a_3$. And since $a_3$ is not a beat point of $X$ we obtain that $U_{b_3}\subsetneq\widehat{U}_{a_3}\subseteq \{b_3\}\cup\mnl(X)$. But $\# U_{b_3}\geq 3$. Thus $\# U_{b_3}= 3$, $\beta_3=2$ and $\widehat{U}_{a_3}=\{b_3\}\cup\mnl(X)$. Hence $\Ss=\mnl(X)\times\mxl(X)$. Thus
\begin{displaymath}
\chi(X)=6-\#\Ss+\sum\limits_{j=1}^3 (\alpha_j-1)(\beta_j-1)=6-9+3=0
\end{displaymath}
which entails a contradiction. This proves claim 1.

Hence $X-\mnl(X)$ is homeomorphic to the following space
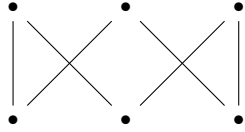
\begin{displaymath}
\begin{tikzpicture}[x=3cm,y=3cm]
	\tikzstyle{every node}=[font=\footnotesize]
	
	\foreach \x in {1,...,3} \draw (0.5*\x,1) node(a\x){$\bullet$};
	\foreach \x in {1,...,3} \draw (0.5*\x,0.5) node(b\x){$\bullet$};
	
	\foreach \x in {1,2} \draw (a1)--(b\x);
	\foreach \x in {1,3} \draw (a2)--(b\x);
	\foreach \x in {2,3} \draw (a3)--(b\x);
	
	\end{tikzpicture}
\end{displaymath}

\underline{Claim 2}: If $k,l\in\{1,2,3\}$ are distinct elements such that $\beta_k=\beta_l=2$ then $\#(\widehat{U}_{b_k}\cap\widehat{U}_{b_l})=1$.

Suppose that there exist distinct elements $k,l\in\{1,2,3\}$ such that $\beta_k=\beta_l=2$ and $\#(\widehat{U}_{b_k}\cap\widehat{U}_{b_l})=2$. Let $m$ be the remaining element of $\{1,2,3\}$ and let $c_1$ and $c_2$ be the elements of $\widehat{U}_{b_k}\cap\widehat{U}_{b_l}$. Note that $\mxl(X)\subseteq F_{c_1}\cap F_{c_2}$. If $\{c_1,c_2\}\subseteq \widehat{U}_{b_m}$ then $|\K(X-\{c_3\})|$ is homeomorphic to $S^2$ contradicting \ref{lemma_subspace_with_non_trivial_H2}. Thus $\{c_1,c_2\}\nsubseteq \widehat{U}_{b_m}$. Hence $\beta_m=2$ and $c_3\in U_{b_m}$. If either $c_3<b_k$ or $c_3<b_l$ then $\mxl(X)\subseteq F_{c_3}$. Otherwise, since $c_3$ is not a beat point of $X$ we obtain that $F_{b_m}\subsetneq\widehat{F}_{c_3}\subseteq \{b_m\}\cup\mnl(X)$. As $\# F_{b_m}=3$, it follows that $\widehat{F}_{c_3}=\{b_m\}\cup\mnl(X)$ and hence $\mxl(X)\subseteq F_{c_3}$. Thus $\mxl(X)\subseteq F_{c_3}$ in any case.

Hence $\Ss=\mnl(X)\times\mxl(X)$. Thus 
\begin{displaymath}
\chi(X)=6-\#\Ss+\sum_{j=1}^3 (\alpha_j-1)(\beta_j-1)=6-9+3=0
\end{displaymath}
which entails a contradiction. This proves claim 2.

Now, let $a_1$, $a_2$ and $a_3$ be the only elements of $\widehat{F}_{b_1}\cap\widehat{F}_{b_2}$, $\widehat{F}_{b_1}\cap \widehat{F}_{b_3}$ and $\widehat{F}_{b_2}\cap \widehat{F}_{b_3}$ respectively. Without loss of generality we may assume that $\beta_1=\beta_3=2$. Thus $\#(\widehat{U}_{b_1}\cap\widehat{U}_{b_3})=1$ by claim 2. Let $c_1$, $c_2$ and $c_3$ be the only elements of $\widehat{U}_{b_1}-\widehat{U}_{b_3}$, $\widehat{U}_{b_1}\cap \widehat{U}_{b_3}$ and $\widehat{U}_{b_3}-\widehat{U}_{b_1}$ respectively.

If $\beta_2=2$ then $\#(\widehat{U}_{b_2}\cap\widehat{U}_{b_1})=1$ and $\#(\widehat{U}_{b_2}\cap\widehat{U}_{b_3})=1$ by claim 2. Then $X$ is homeomorphic to the following space
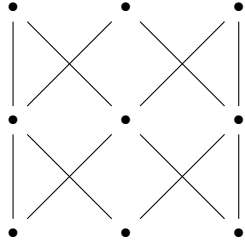
\begin{displaymath}
\begin{tikzpicture}[x=3cm,y=3cm]
	\tikzstyle{every node}=[font=\footnotesize]
	
	\foreach \x in {1,...,3} \draw (0.5*\x,1) node(a\x){$\bullet$};
	\foreach \x in {1,...,3} \draw (0.5*\x,0.5) node(b\x){$\bullet$};
	\foreach \x in {1,...,3} \draw (0.5*\x,0) node(c\x){$\bullet$};
	
	\foreach \x in {1,2} \draw (a1)--(b\x);
	\foreach \x in {1,3} \draw (a2)--(b\x);
	\foreach \x in {2,3} \draw (a3)--(b\x);
		
	\foreach \x in {1,2} \draw (c1)--(b\x);
	\foreach \x in {1,3} \draw (c2)--(b\x);
	\foreach \x in {2,3} \draw (c3)--(b\x); 
\end{tikzpicture}
\end{displaymath}
Hence $|\K(X)|$ is homotopy equivalent to $S^1$ and then $X$ is not homotopically trivial.

Thus $\beta_2=3$ and hence $X$ is homeomorphic to the space of figure \ref{fig_ht_non-contractible}.
\end{proof}

\bibliographystyle{acm}

\bibliography{references}

\end{document}